\documentclass{amsart}

\usepackage[all]{xy}
\usepackage{graphicx}
 \usepackage{verbatim}
\usepackage[danish,english]{babel}
\usepackage{color}
\usepackage{enumerate}
\usepackage[abs]{overpic}
\usepackage[makeroom]{cancel}
 \usepackage{amssymb}

\newcommand{\C}{\mathbb C}

\newcommand{\F}{\mathbb F}
\newcommand{\R}{\mathbb R}
\newcommand{\N}{\mathbb N}

\newcommand{\T}{\mathbb T}

\newcommand{\fE}{\mathcal E}
\newcommand{\fF}{\mathcal F}

\newcommand{\fJ}{\mathcal J}

\newcommand{\fO}{\mathcal O}

\newcommand{\fU}{\mathcal U}
\newcommand{\fV}{\mathcal V}

\newcommand{\supp}{\mathrm{supp }}

\newcommand{\sumi}{\sum_{i=1}^n}
\newcommand{\QT}{\mathrm{QT}}

\newtheorem{thm}{Theorem}[section]
\newtheorem{pro}[thm]{Proposition}
\newtheorem{cor}[thm]{Corollary}
\newtheorem{lem}[thm]{Lemma}

\theoremstyle{definition}

\newtheorem{rem}[thm]{Remark}       
\newtheorem{defn}[thm]{Definition}  

\reversemarginpar

\renewcommand{\qed}{\hfill$\square$}
\numberwithin{equation}{section}

 \newcommand{\mystar}{\mathop{*}}

\begin{document}

\title[Quasitraces on exact C*-algebras are traces]
 {Quasitraces on exact C*-algebras are traces}

\author{ Uffe Haagerup}
\thanks{}

\address{Department of Mathematical Sciences, University of Copenhagen, Universitetsparken 5,
2100 K\o benhavn \O ,
 Denmark.}

\email{haagerup@math.ku.dk}


\date{March 29, 2014}

\thanks{Supported by ERC Advanced Grant no. OAFPG 247321, the Danish Natural Science Research Council, and the Danish National Research Foundation through the Centre for Symmetry and Deformation (DNRF92).}

\begin{abstract}
It is shown that all 2-quasitraces on a unital exact $C^*$-algebra are traces.
As consequences one gets: (1) Every stably finite exact  unital $C^*$-algebra  has a tracial state, and (2) if an $AW^*$-factor of type $II_1$ is generated (as an $AW^*$-algebra) by an exact $C^*$-subalgebra, then it is a von Neumann $II_1$-factor. This is a partial solution to a well known problem of Kaplansky. The present result was used by Blackadar, Kumjian and R\o rdam to prove that $RR(A)=0$ for every simple non-commutative torus of any dimension.
\end{abstract}

\maketitle

\section{Introduction}
Let $A$ be a unital $C^*$-algebra. A 1-quasitrace $\tau$ on $A$ is a function $\tau:A\to\C$ that satisfies
\begin{itemize}
  \item [($i$)] $\tau(x^*x)=\tau(xx^*)\ge 0$,\,\, $x\in A$,
  \item [($ii$)] $\tau(a+ib)=\tau(a)+i\tau(b)$ for $a,b\in A_{\text{sa}}$,
  \item [($iii$)] $\tau$ is linear  on every abelian $C^*$-subalgebra $B$ of $A$.
\end{itemize}

Furthermore, $\tau$ is called a $n$-quasitrace $(n\ge 2)$ if there exists a $1$-quasitrace  $\tau_n$ on $M_n(A)=A\otimes M_n(\C)$ such that
\begin{itemize}
  \item [($iv$)] $\tau(x)=\tau_n(x\otimes e_{11})$,\,\,  $x\in A$,
\end{itemize}
where $(e_{ij})_{i,j=1}^n$ denote the matrix units of $M_n(\C)$. Moreover $\tau$ is called normalized if $\tau(1)=1$.
The famous and still open problem of Kaplansky, asking whether every $AW^*$-factor $M$ of Type $II_1$ is a von Neumann algebra, has an affirmative answer if and only if the unique normalized 1-quasitrace $\tau$ on $M$ is linear (i.e.~$\tau$ is a trace on  $M$).

Blackadar and Handelman proved in $1982$ (cf.~\cite{BH}), that a 2-quasitrace  on a unital $C^*$-algebra $A$ is automatically an $n$-quasitrace for all $n\in\N$. Moreover, they showed that if Kaplansky's problem has an affirmative answer, then every  2-quasitrace on a unital $C^*$-algebra $A$ is a trace.

The main result of this paper is that every 2-quasitrace on a unital exact $C^*$-algebra $A$ is a trace. In particular, this holds for every unital nuclear $C^*$-algebra  $A$ and every unital $C^*$-subalgebra  $A$ of a nuclear $C^*$-algebra.

Quasitraces (or more precisely 2-quasitraces) have become an important tool in the classification theory of $C^*$-algebras. It was proved in \cite{BH} that every stably finite unital $C^*$-algebra has a 2-quasitrace. More generally, Blackadar and R\o rdam proved in \cite{BR} that for a unital $C^*$-algebra $A$, every state on $K_0(A)$ is given by a 2-quasitrace.  The latter result combined with the main result of this paper was used in \cite{BKR} to prove that every simple non-commutative torus has real rank 0.

A first draft of this paper was completed in 1991 and it has been circulated among experts in the field since then. Thanks to the efforts of George Elliott, the paper has now finally been completed for publication.

The remaining sections of this paper have been left essentially unchanged since the first draft from 1991. We will use the rest of this section to mention briefly four more recent results related to this paper:

In 1997 (cf.~\cite{Kir4}) Kirchberg used the main result of this paper to prove that every exact stably projectionless $C^*$-algebra $A$ has a nontrivial lower semicontinuous trace on its Pedersen ideal $P(A)$.

In 1999 (cf.~\cite{HT}), Thorbj\o rnsen and the author used random matrix methods to give new proofs, not relying on quasitraces or $AW^*$-algebras, for the facts that $(1)$ a simple unital stably finite exact $C^*$-algebra has a tracial state, and (2) every state on $K_0(A)$ for a unital exact $C^*$-algebra $A$ is given by a tracial state.

In 2006 (cf.~\cite{Kir}) Kirchberg proved that there exists a unital $C^*$-algebra $A$ with a 1-quasitrace $\tau$ which is not a 2-quasitrace.
In particular $\tau$ is not a trace.

In 2011 (cf.~\cite{BW}) Brown and Winter gave a short proof of the main result of this paper for the special case where $A$ is a nuclear $C^*$-algebra of finite nuclear dimension.

\section{An application of Voiculescu's semicircular system}

We shall need the following algebraic characterization of unital $C^*$-algebras without trace states:

\begin{lem}\label{l:2.1}
Let $A$ be a unital $C^*$-algebra. Then the following two conditions are equivalent:
\begin{itemize}
  \item [($a$)] $A$ has no trace state.
  \item [($b$)] There is a finite set $\{a_1,\ldots,a_n\}\subset A$ such that
  $$\sum_{i=1}^{n} a_i^*a_i=1\qquad\text{and}\qquad ||\sum_{i=1}^{n} a_ia_i^*||<1.$$
\end{itemize}
\end{lem}

\begin{proof}
$(b)\Rightarrow (a)$: Assume $(b)$ and let $\tau$ be a trace state on $A$. Then $\tau(\sum_{i=1}^n a_ia_i^*)=\tau(\sum_{i=1}^n a_i^* a_i)=1$ which contradicts that $||\sum_{i=1}^n a_i a_i^*|| <1$.\\
$(a)\Rightarrow(b)$: Assume $(a)$. Then the second dual $A^{**}$ is a von Neumann algebra without normal trace states; i.e.~$A^{**}$ is a properly infinite von Neumann algebra. Hence, we can choose two isometries $v_1,v_2\in A^{**}$ such that $v_1^*v_1\bot v_2^* v_2$ and $v_1v_1^*+v_2v_2^*=1$.
Choose a net $(b_1^{(\alpha)}, b_2^{(\alpha)})_{\alpha\in \Lambda}$ in $A\oplus A$ which converges to $(v_1,v_2)$ in the $\sigma$-strong*-topology.
Then,

\begin{eqnarray*}
&&\sum_{i=1}^2 (b_i^{(\alpha)})^*b_i^{(\alpha)}\to \sum_{i=1}^2 v_i^*v_i=2 \quad (\sigma\text{-weakly})\\
&&\sum_{i=1}^2 b_i^{(\alpha)}(b_i^{(\alpha)})^*\to \sum_{i=1}^2 v_iv_i^*=1\quad (\sigma\text{-weakly}).
 \end{eqnarray*}

Since the restriction of $\sigma$-weak topology on $A^{**}$ to $A$ is equal to the $\sigma(A,A^*)$-topology we get
$$\{2,1\}  \in \overline{\{( \sum_{i=1}^2 b_i^* b_i, \sum_{i=1}^2 b_i b_i^*) \mid b_1,b_2\in A\}}^{\sigma(A\oplus A, A^*\oplus A^*)}.$$

Since the set
$$\{ (\sum_{i=1}^n b_i^*b_i,\sum_{i=1}^n b_i b_i^*)\mid n\in\N, b_1,\ldots, b_n\in A
\}$$
is convex, and since convex sets in Banach spaces have the same closure in the norm and  weak topologies, we get that for all $\varepsilon>0$ there is an $n\in\N$ and $b_1\ldots,b_n\in A$ such that
\begin{eqnarray*}
  ||\sum_{i=1}^n b_i^*b_i-2||<\varepsilon\\
  ||\sum_{i=1}^n b_ib_i^*-1||<\varepsilon.
\end{eqnarray*}
Assume $\varepsilon=\frac 13$. Then $\frac 5 3 \le \sum_{i=1}^n b_i^* b_i\le \frac 7 3 $ and $\frac 2 3 \le \sum_{i=1}^n b_i b_i^*\le \frac 4 3$.
Set $$a_i=b_i(\sum_{i=1}^n b_i^* b_i)^{-1/2}.$$
Then $\sum_{i=1}^n a_i^*a_i=1$ and since
$$a_ia_i^*=b_i(\sum_{j=1}^n b_j^*b_j)^{-1} b_i^*\le \frac 3 5 b_i b_i^*$$
we have
$$||\sum_{i=1}^n a_ia_i^*||\le \frac 3 5 || \sum_{i=1}^n b_ib_i^*||\le \frac 4 5 <1,$$
which proves $(b)$.
\end{proof}

\begin{rem}
   By using an arbitrary number of isometries $(v_i)_{i=1}^r$ in the proof of $(a)\Rightarrow (b)$ one gets easily $(a)\iff (b')$, where
   \begin{itemize}
     \item [($b'$)] For all $\varepsilon>0$ there is a finite set $\{a_1,\ldots,a_n\}\subset A$, such that
     $$\sum_{i=1}^n a_i^*a_i=1 \qquad \text{and} \qquad ||\sum_{i=1}^n a_i a_i^*||< \varepsilon.$$
   \end{itemize}
\end{rem}

In \cite{V1} Voiculescu introduced the reduced  free product $(A,\phi)=\mystar_{i\in I} (A_i, \phi_i)$ of a family $(A_i)_{i\in I}$ of unital $C^*$-algebras with respect to a specified set of states $(\phi_i)_{i\in I}$, $\phi_i\in S(A_i)$.
$$\phi=\mystar_{i\in I} \phi_i$$
is a state on $A$ characterized by
$$\phi(a_{1}a_{2}\cdots a_{n})=0$$
whenever $a_k\in A_{i_k}$, $i_1\ne i_2\ne \cdots\ne i_n$ and $\phi_{i_k}(a_k)=0$.
A special case of interest is the semicircular system introduced in \cite{V2}. Here
\begin{eqnarray*}
  &&A_i=C([-1,1])\\
  &&\tau_i(f)=\frac 2 \pi \int_{-1}^{1} f(t)\sqrt{1-t^2}\,dt, \qquad f\in C([-1,1]),
\end{eqnarray*}
for all $i$. Let $x_i\in A_i\subset A$ be the identity function on $[-1,1]$. Then $A$ is the $C^*$-algebra generated by $1$ and $(x_i)_{i\in I}$, and $\tau=\mystar_{i\in I} \tau_i$ is a faithful trace on $A$, and $(A,(x_i)_{in\in I},\tau)$ is a semicircular system in the sense of \cite{V2}.
A concrete model of $(A,(x_i)_{i\in I},\tau)$ can be obtained in the following way (cf.~\cite {V1}):
Let $H$ be a Hilbert space with orthonormal basis  $(e_i)_{i\in I}$, and let
$\fF(H)=\C\oplus (\bigoplus_{n=1}^\infty H^{\otimes n})$  be the full Fock space based on $H$.
Let $s_i$ be the isometry $\fF(H)\to \fF(H)$ obtained by tensoring from the left by $e_i$ on each $H^{\otimes n}$, $n=0,1,2,\ldots$,
where $H^{\otimes n}=\C$ for $n=0$. Then
\begin{eqnarray}\label{e:page2.5}
&&  s_i^*s_i=1\quad \forall i\in I\\
\nonumber&&  s_is_i^* \perp s_js_j^* \quad \forall i,j\in I, i\ne j\\
\nonumber&&  1-\sum_{i\in I} s_is_i^* \text{ is the projection of $\fF(H)$ onto the $\C$-component of } \fF(H).
\end{eqnarray}

Moreover, $x_i=\frac 1 2 (s_i+s_i^*)$, $i\in I$, is a semicircular system and the trace state $\tau$ is simply the ``vacuum-state'', i.e., the vector state given by a unit vector in the $\C$-component of $\fF(H)$ on $A=C^*(\{(x_i)_{i\in I}, 1\})$.
If $I=\{1,\ldots,n\}$ (resp.~$I=\N$) we will denote the unital $C^*$-algebra  generated by the $x_i$'s by $\fV_n$ (resp.~$\fV_\infty$). By the equations in (\ref{e:page2.5}), one has natural embeddings
\begin{eqnarray*}
 &&\fV_n\hookrightarrow \fE_n, \quad n\in\N\\
 && \fV_\infty \hookrightarrow \fO_\infty,
\end{eqnarray*}
where $\fE_n$ denotes the compact extension of the Cuntz  algebra $\fO_n$  given in \cite{C1}, and $\fO_\infty$  is the usual Cuntz algebra generated by a sequence of isometries $(s_i)_{i=1}^\infty$. The trace $\tau$ on $\fV_n$, $n\in\N$ (resp.~$\fV_\infty$) is the restriction of the unique (pure) state $\phi$ on $\fE_n$ (resp.~$\fO_\infty$) which satisfies $\phi(s_is_i^*)=0$ for all $i\in I$.

\begin{lem}\label{l:2.3}
  Let $A$ be a unital $C^*$-algebra without trace states. Then $A\otimes \fV_\infty$ contains a non-unitary isometry. ($A\otimes \fV_\infty$ is the spatial\,(=minimal)  $C^*$-tensor product of $A$ and $\fV_\infty$.)
\end{lem}
\begin{proof}
  Choose by Lemma \ref{l:2.1} elements $a_1,\ldots, a_n\in A$ such that
  $$\sumi a_i^*a_i=1\qquad\text{and}\qquad ||\sumi a_ia_i^*||<1$$
  and let $(s_i)_{i=1}^\infty $ be as in the equations in (\ref{e:page2.5}) with $I=\N$. Then
  $\fO_\infty=C^*((s_i)_{i=1}^\infty)$. With  $x_i=\frac 12 (s_i+s_i^*)$, $i\in\N$,
  $$\fV_\infty=C^*((x_i)_{i=1}^\infty,1),$$
  and $(x_i)_{i=1}^\infty$ is a semicircular system with respect to a faithful trace state $\tau$ on $\fV_\infty$.
  With the above notation,
  $$A\otimes \fV_\infty\subset A\otimes \fO_\infty.$$
  Let $y=2\sum_{i=1}^n a_i\otimes x_i \in A\otimes \fV_\infty$. Then $y=v+w$, where $v,w\in A\otimes \fO_\infty$ are given by
  $$v=\sumi a_i\otimes s_i,\qquad w=\sumi a_i\otimes s_i^*.$$
  Since $$s_i^*s_i=\left\{
                    \begin{array}{ll}
                      1, & \hbox{$i=j$} \\
                      0, & \hbox{$i\ne j$}
                    \end{array}
                  \right.,
                  $$
  we have $v^*v=\sumi a_i^*a_i\otimes 1=1_{A\otimes \fO_\infty}$,
  i.e.~$v$ is an isometry. The range projection of $v$ is clearly bounded by $1\otimes \sumi s_is_i^*$.
  Thus $v$ is a non-unitary isometry in $A\otimes \fO_\infty$.
  Since
  $$w^*=\sumi a_i^*\otimes s_i$$
  we get as above
  $$ww^*=\sumi a_ia_i^* \otimes 1,$$
  so by the choice of $(a_i)_{i=1}^n$, we have $||w||<1$.
  We may assume that $A\otimes \fO_\infty\subset B(K)$ for some Hilbert space $K$ (unit preserving embedding).
  For $\xi\in K$,
  $$||y\xi||=||(u+w)\xi||\ge ||v\xi||-||w\xi||\ge (1-||w||)||\xi||.$$
Hence $y^*y$ is invertible. Note that
$$y=(1+w v^*)v.$$
  Moreover, $1+w v^*$ is invertible because $||wv^*||<1$, while $v$ is not invertible. Hence $y$ is not invertible.
  Let $u=y(y^*y)^{-1/2}\in A\otimes \fV_\infty$. Then
  $$y=u(y^*y)^{1/2}$$
  is the polar decomposition of $y$, and $u^*u=1$ while $uu^*\ne 1$. This completes the proof.

\end{proof}

\begin{thm}\label{t:2.4}
  Let $A$ be a unital $C^*$-algebra without trace states.
  Then $A\otimes C_r^*(\F_\infty)$ contains a non-unitary isometry.
  Here $\F_\infty$ is the free group on infinitely (countable) many generators.
\end{thm}

\begin{proof}
  Let $(u_n)_{n=1}^\infty$ be the unitary generators of $C_r^*(\F_\infty)$, and let
  $$y_n=\frac 1 2 (u_n+u_n^*).$$
  Then $(y_n)_{n=1}^\infty$ is a free system in the sense of Voiculescu, $\sigma(y_n)=[-1,1]$, and the measure on the spectrum $\sigma(y_n)$ given by the trace has density
  $$g(t)=\frac 1{\pi \sqrt{1-t^2}},\qquad t\in (-1,1)$$
  obtained by projecting the uniform distribution on the unit circle $\T$ onto the real line.
  Let
  $$ G(t)=\int_{-1}^t g(t)\, dt,\qquad t\in[-1,1]$$
  be the distribution function for $g(t)$ and let
  $$F(t)=\int_{-1}^t \frac 2\pi \sqrt{1-t^2}\,dt,\qquad t\in[-1,1]$$
  be the distribution function for the semicircular distribution.  Then $\Phi=F^{-1}\circ G$ is a homeomorphism of $[-1,1]$ onto itself which transforms the measure given by the density $g(t)$ onto the measure with density
  $$f(t)=\frac 2 \pi \sqrt{1-t^2}, \quad 1< t<1.$$
  Hence, $x_n=\Phi(y_n)$ form a semicircular system in the sense of \cite{V2}, which shows that
  $$\fV_\infty \simeq C^*(1,x_1,x_2,\ldots)\subset C_r^*(\F_\infty).$$
  This extends to a unital embedding of $A\otimes \fV_\infty$ into $A\otimes C_r(\F_\infty)$ and the theorem now follows from Lemma \ref{l:2.3}.

\end{proof}

\begin{rem}\label{r:2.5}
 ($a$) Since any free group $\F_n$ with at least 2 generators contains a copy of $\F_\infty$, $C_r^*(\F_\infty)$ has a unital embedding in $C_r^*(\F_n)$, $n\ge 2$, so in Theorem \ref{t:2.4} $C_r^*(\F_\infty)$ can be replaced by $C_r^*(\F_n)$ for any $n\ge 2$.\\
  ($b$) By choosing a continuous function $[-1,1]\to \T$ which transforms the semicircular distribution into the uniform distribution on $\T$ one gets that $C_r^*(\F_n)$ can be embedded in $\fV_n$ for any $n\ge 2$. Hence in Lemma \ref{l:2.3}, $\fV_\infty$ can be replaced by $\fV_n$ for any $n\ge 2$.
\end{rem}


\section{Quasitraces on $C^*$-algebras and $AW^*$-algebras}
Throughout this section $A$ denotes a unital $C^*$-algebra.
It has become customary to rename the 2-quasitraces of Blackadar and Handelman \cite{BH} to quasitraces (see for example \cite{R},\cite{BKR}).
Hence we will use the following definition.
\begin{defn}\label{d:3.1}
  A quasitrace $\tau$ on $A$ is a function $\tau:A\to\C$ which satisfies:
\begin{itemize}
\item[($i$)] $\tau(x^*x)=\tau(xx^*)\ge 0$ for all $x\in A$.
\item[($ii$)] $\tau(a+bi)=\tau(a)+i\tau(b)$ for $a,b\in A_{\mathrm{sa}}$.
\item[($iii$)] $\tau$ is linear on every abelian $C^*$-subalgebra $B$ of $A$.
\item[($iv$)] There is a function $\tau_2:M_2(A)\to\C$ satisfying $(i)$, $(ii)$, $(iii)$ such that
$$\tau(x)=\tau_2(x\otimes e_{11}),\quad x\in A.$$
\end{itemize}
\end{defn}

A quasitrace is normalized if $\tau(1)=1$ and the set of normalized quasitraces on $A$ is denoted by $\QT(A)$.

Note that $(i)$, $(ii)$, $(iii)$ correspond to the definition of quasitraces in \cite{BH}. If $A$ is an $AW^*$-algebra, $(i)$, $(ii)$, and $(iii)$ imply $(iv)$, but this is not true in general (cf.~\cite{Kir}).

By Theorem II.2.2 in \cite{BH} there is a bijection between $\QT(A)$ and the set of lower continuous semicontinuous dimension functions $D$ on $A$ (in the sense of Cuntz \cite{C2}). The correspondence is given by
$$D(x)=\sup\limits_{\varepsilon>0} \tau(f_\varepsilon(|x|)),\quad x\in A$$
where
$$f_\varepsilon(t)=\left\{
                     \begin{array}{ll}
                       0, & \hbox{$0\le t\le \tfrac \varepsilon 2$} \\
                       \tfrac 2{\varepsilon}t-1, & \hbox{$\tfrac{\varepsilon} 2< t<\varepsilon $} \\
                       1, & \hbox{$t\ge \varepsilon$}
                     \end{array}
                   \right.
$$

This correspondence together with Theorem I.1.1 in \cite{BH} gives
\begin{pro}\label{p:3.2}
  Let $\tau$ be a quasitrace on $A$. Then
  $$I=\{x\in A\mid \tau(x^*x)=0\}$$
  is a closed two-sided ideal in $A$, and there is a (unique) quasitrace $\bar\tau$ on $A/I$ such that
  $$\tau(x)=\bar\tau(\rho(x)),\quad x\in A,$$
  where $\rho$ denotes the quotient map.
\end{pro}

By an ultrapower construction based on dimension functions Blackadar and Handelman showed that all quasitraces come from $AW^*$-algebras in the following sense:
\begin{pro}[\cite{BH} Corollary II.2.4]\label{p:3.3}
  Let $\tau$ be a quasitrace on $A$. There exists a unital $*$-homomorphism $\theta$ of $A$ into a finite $AW^*$-algebra and a quasitrace $\bar \tau$ on $M$ such that
  $$\tau(a)=\theta \circ \bar \tau(a),\quad a\in A.$$
\end{pro}

By well known properties for quasitraces of $AW^*$-algebras, it follows that
\begin{cor}[\cite{BH} Section II]\label{c:3.4}
  Let $\tau$ be a quasitrace on $A$. Then
  \begin{itemize}
    \item [($1$)] $\tau$ is order preserving on $A_{\text{sa}},$
    \item [($2$)] $\tau$ extends uniquely to a quasitrace $\tau_n$ on $M_n(A)$ such that $\tau_n(x\otimes e_{11})=\tau(x)$, $x\in A$ for all $n\in N$.
  \end{itemize}
\end{cor}

\begin{lem}\label{l:3.5}
  Let $\tau$ be a quasitrace on $A$, and let
  $$||x||_2=\tau(x^*x)^{1/2}, \qquad x\in A.$$
\end{lem}
Then
\begin{itemize}
  \item[($1$)] $\tau(a+b)^{1/2}\le \tau(a)^{1/2}+\tau(b)^{1/2}$, $a,b\in A_+$.
  \item[($2$)] $||x+y||_2^{2/3}\le ||x||_2^{2/3}+||y||_2^{2/3}$, $x,y\in A$.
  \item[($3$)] $||xy||_2\le ||x||||y||_2$ and $||xy||_2\le ||x||_2||y||$, $x,y\in A$.
\end{itemize}
\begin{proof}
  ($1$) follows by a slight modification of the proof of Corollary II.1.11 in \cite{BH}:
  Let
  $$x=a^{1/2}\otimes e_{11}+b^{1/2}\otimes e_{21}\in M_2(A).$$
  Then
  $$x^*x=\left(
           \begin{array}{cc}
             a+b & 0 \\
             0 & 0 \\
           \end{array}
         \right),
         \qquad
xx^*=\left(
  \begin{array}{cc}
    a & a^{1/2}b^{1/2} \\
    b^{1/2}a^{1/2} & b \\
  \end{array}
\right).
$$
Moreover, for $\lambda >0$, let
$$x_\lambda=\lambda^{1/2} a^{1/2}\otimes e_{11}-\lambda^{-1/2}b^{1/2}\otimes e_{21}.$$
Then,
$$xx^*\le x^*+x_\lambda x_\lambda^*=\left(
                                     \begin{array}{cc}
                                       (1+\lambda)a & 0 \\
                                       0 & (1+\lambda^{-1})b \\
                                     \end{array}
                                   \right).
$$
Hence by $(i)$ and $(iv)$ of Definition \ref{d:3.1} and Corollary \ref{c:3.4},
$$
\tau(a+b)=\tau_2(x^*x)=\tau_2(x^*x)=\tau_2(xx^*)\le \tau_2\left(
                                                            \begin{array}{cc}
                                                              (1+\lambda)a & 0 \\
                                                              0 & (1+\lambda^{-1})b \\
                                                            \end{array}
                                                          \right).
$$

Since $a\otimes e_{11}$ and $b\otimes e_{22}$ commute in $M_2(A)$, we get by $(ii)$ and $(iv)$ of Definition \ref{d:3.1}
\begin{eqnarray}\label{e:page3.4}
\tau(a+b)&\le&(1+\lambda)\tau_2(a\otimes e_{11})+(1+\tfrac 1{\lambda})\tau(b\otimes e_{22})\\
\nonumber&=&(1+\lambda) \tau(a)+ (1+\tfrac 1 \lambda) \tau(b).
\end{eqnarray}
The last equality follows because with $y=b^{1/2}\otimes e_{12}$, we have $\tau(b\otimes e_{22})=\tau(y^*y)=\tau(yy^*)=\tau(b\otimes e_{11})$.
If $\tau(a)>0$ and $\tau(b)>0$ the right hand side of Equation (\ref{e:page3.4}) attains its minimum at $\lambda=(\tau(b)/\tau(a))^{1/2}$ and the minimum value is $(\tau(a)^{1/2}+\tau(b)^{1/2})^2$, proving $(1)$ in this case.
If $\tau(a)=0$ (resp.~$\tau(b)=0$), then $(1)$ follows by letting $\lambda\to\infty$ (resp.~$\lambda\to 0$).

(2) Let $x,y\in A$. For $\lambda>0$,
\begin{eqnarray*}
  (x+y)^*(x+y)&\le& (x+y)^*(x+y)+(\lambda^{1/2}x-\lambda^{-1/2} y)^*(\lambda^{1/2}x-\lambda^{-1/2}y)\\
  &=& (1+\lambda) x^*x +(1+\frac 1{\lambda}) y^*y.
\end{eqnarray*}

Hence by $(1)$:
  $$||x+y||_2 \le (1+\lambda)^{1/2}||x||_2+(1+\tfrac 1 \lambda)^{1/2} ||y||_2.$$
If $||x||_2>0$ and $||y||_2>0$ the right side attains its minimum at $\lambda=(||y||_2/||x||_2)^{3/2}$ and the minimum value is
$(||x||_2^{2/3}+||y||_2^{3/2})^{3/2})$, proving ($2$) in this case.
The remaining cases $||x||_2=0$ and $||y||_2=0$ follow by letting $\lambda\to \infty$ and $\lambda\to 0$.

($3$) Since $y^*x^*xy\le ||x||^2 y^*y$ the first inequality follows from Corollary \ref{c:3.4}, and the second inequality now follows by using $||z||_2=||z^*||_2$, $z\in A$.
\end{proof}

\begin{defn}\label{d:3.6}
If $\tau$ is a faithful quasitrace on $A$, we let
$$d_\tau(x,y)=||x-y||_2^{2/3}, \quad x,y\in A.$$
By Lemma \ref{l:3.5},  $d_\tau$ is a metric on A.
\end{defn}

\begin{lem}\label{l:3.7}
Let $\tau$ be a faithful quasitrace on $A$. Then
\begin{itemize}
  \item [($1$)] The involution $x\to x^*$ is continuous in the $d_\tau$-metric.
  \item [($2$)] The sum is continuous in the $d_\tau$-metric on $A$.
  \item [($3$)] The product is continuous in the $d_\tau$-metric on bounded sets of $A$.
  \item [($4$)] $x\to \tau(x)$ is continuous in the $d_\tau$-metric on $A_+$.
\end{itemize}
\end{lem}

\begin{proof}
($1$)  Clear since $||x||_2=||x^*||_2$, $x\in A$.
($2$) Clear from Lemma \ref{l:3.5}.
($3$) For $x,y,x_0,y_0\in A$, Lemma \ref{l:3.5}(2) and (3) give
\begin{eqnarray*}
  ||xy-x_0y_0||_2^{2/3}&\le& ||x(y-y_0)||_2^{2/3}+||(x-x_0)y_0||_2^{2/3}\\
  &\le& ||x||^{2/3}||y-y_0||_2^{2/3}+||y_0||^{2/3}||x-x_0||_2^{2/3}.
\end{eqnarray*}
This proves $(3)$.
(4) For $a,b\in A_+$
$$a\le b+|a-b| \quad\text{and}\quad b\le a+|a-b|$$
Then by Lemma \ref{l:3.5}(1)
$$|\tau(a)^{1/2}-\tau(b)^{1/2}|\le \tau(|a-b|)^{1/2},$$
and since $\tau$ is linear on $C^*(|a-b|,1)$ we get
$$|\tau(a)^{1/2}-\tau(b)^{1/2}|\le \tau(|a-b|^2)^{1/4}\tau(1)^{1/4}=||a-b||_2^{1/2}\tau(1)^{1/4}.$$
This proves (4).
\end{proof}
\begin{lem}\label{l:3.8}
  Let $\tau$ be a faithful quasitrace on $A$. Then the unit ball of $A$ is closed in the $d_\tau$-metric.
\end{lem}

\begin{proof}
  Let $x_n$ be a sequence in the unit ball of $A$ converging to $x\in A$ in the $d_\tau$-metric.
  Let $a_n=x^*_nx_n$ and $a=x^*x$.
  By Lemma \ref{l:3.7},
\begin{equation}\label{e:page3.6}
  \tau(a_n^p)\to \tau(a^p),\quad p=0,1,2,\ldots.
\end{equation}
  Let $\mu_n$ (respectively $\mu$) be the measure on the spectrum $\sigma(a_n)$ (resp.~$\sigma(a)$) given by  the linear functional
  $\tau|_{C^*(a_n,1)}$ (resp.~$\tau|_{C^*(a,1)}$). All the measures can be considered as measures in the interval $J=[0,\max\{1,||a||\}]$ because $||a_n||\le 1$ for all $n$.
  Hence by Equation (\ref{e:page3.6}), $\mu_n\to \mu$ in the $w^*$-topology on $C(J)^*$.
  Since $\mu_n$  has support in $[0,1]$, $\mu$ also has support in $[0,1]$, and since $\tau$ is faithful $\supp(\mu)=\sigma(a)$. Hence $||x||^2=||a||\le 1$.
\end{proof}

\begin{lem}\label{l:3.9}
  Let $\tau$ be a faithful quasitrace on $A$. If the unit ball of $A$ is complete in the $d_\tau$-metric, then $A$ is an $AW^*$-algebra and $\tau$ is a normal quasitrace on $A$, i.e.
  $$\tau(\bigvee_{i\in I} p_i)=\sum_{i\in I} \tau(p_i)$$
  for any orthogonal set of projections $(p_i)_{i\in I}$ in $A$.
\end{lem}
\begin{proof}
  Let $B$ be a maximal abelian $C^*$-subalgebra of $A$. By Lemma \ref{l:3.8}, the unit ball of $B$ is closed in the $d_\tau$-metric, and hence also complete in the $d_\tau$-metric by the assumptions on $A$. Since $\tau_B$ is a positive linear functional on $B$, $||x-y||_2=d_\tau(x,y)^{3/2}$  is an equivalent metric on $B$, and completeness of the unit ball $B$ in the $||\cdot||_2$-norm associated with the positive faithful functional $\tau$ implies that $B$ is a $W^*$-algebra and that $\tau$ is a normal trace on $B$. This clearly implies that $A$ is an $AW^*$-algebra, and that
  $\tau(\bigvee_{i\in I}p_i)=\sum_{i\in I} \tau(p_i)$
  for every orthogonal set of projections $(p_i)_{i\in I}$ in $A$.
\end{proof}
The converse of Lemma \ref{l:3.9} is also true:
\begin{pro}\label{p:3.10}
  Let $M$ be a finite $AW^*$-algebra with a normal faithful quasitrace $\tau$. Then the unit ball of $M$ is complete in the $d_\tau$-metric.
\end{pro}
\begin{proof}
  Let $e,f\in M$ be projections. Then the Kaplansky identity $e-e\wedge f\sim e\vee f -f$ holds (cf.~Theorem 5.4 of \cite{Kap}).
  Hence, $\tau(e)-\tau(e\wedge f)=\tau(e\vee f)-\tau(f)$, which implies that $\tau(e\vee f)\le \tau(e)+\tau(f)$.
  Therefore, the normality of $\tau$ ensures that
  $$\tau(\bigvee_{n=1}^\infty e_n)\le \sum_{n=1}^\infty \tau(e_n)$$
  for any sequence of projections $(e_n)_{n=1}^\infty $ in $M$.
   Let us prove first that the unitary group $U(M)$ is complete in the $d_\tau$-metric: Let $(u_n)_{n=1}^\infty$ be a Cauchy sequence of unitaries in the $d_\tau$-metric. By passing to a subsequence we may assume that
  $$d_\tau(u_n,u_{n+1})^{3/2}=||u_n-u_{n+1}||_2< 4^{-n}, \quad n\in \N.$$
  Let  $e_n=\chi_{[0,2^{-n}]}(|u_n-u_{n+1}|)$, where $\chi_E$ denotes the characteristic function of an interval $E\subset\R$.
  Then
  $$||(u_n-u_{n+1})e_n||\le 2^{-n}$$
  and since  $e_n^\perp\le 2^n |u_n-u_{n+1}|$,
\begin{eqnarray*}
  \tau(e_n^\perp)&\le& 2^n\tau(|u_n-u_{n+1}|)\\
&\le& 2^n \tau(1)^{1/2}||u_n-u_{n+1}||_2\\
&<& 2^{-n} \tau(1)^{1/2}.
\end{eqnarray*}
  Here we have used the linearity of $\tau$ on $C^*(|u_n-u_{n+1}|,1)$.

Let $f_n=\bigwedge_{k\ge n} e_n$. Then
$$\tau(f_n^\perp)\le \sum_{k=n}^\infty \tau(e_k^\perp)<2^{1-n}\tau(1)^{1/2}.$$
For all $k\ge n$,
$$||(u_k-u_{k+1})f_n)||\le ||(u_k-u_{k+1} )e_k|| \le 2^{-k}.$$

Hence $\sum_{k=n}^\infty (u_{k+1}-u_k)f_n$ converges in the $C^*$-norm.
Therefore,
\begin{equation}\label{e:page3.9}
  v_n=\lim\limits_{k\to\infty} u_k f_n
\end{equation}
exists in the $C^*$-norm for all $n$. Moreover,
$$v_n^*v_n=\lim_{k\to\infty} f_nu_k^*u_kf_n=f_n.$$
Therefore, $v_n$ is a partial isometry, and since $f_1\le f_2\le \ldots$ we have from the equation in (\ref{e:page3.9}) that
$$v_nf_m=v_m, \quad n\ge m.$$
Let $v_0=0$ and $f_0=0$. Then $w_n=v_n-v_{n-1}$ is a sequence of partial isometries with orthogonal supports and orthogonal ranges. By Lemma 20 of \cite{Kap2}, there is a partial isometry $w\in M$ such that
$$w_n=w(f_n-f_{n-1})\quad \text{ for all } n\in N,$$
and such that
$$w^*w=\bigvee_{n\in\N} (w_n^*w_n),\qquad ww^*=\bigvee_{n\in\N} (w_nw_n^*).$$
Since $w_n^*w_n=f_n-f_{n-1}$ and since $\tau(f_n^\perp)\to 0$ for $n\to\infty$, $w^*w=1$. Hence  $ww^*=1$ by finiteness of $M$.
Note that $v_n=\sum_{k=1}^n (v_k-v_{k-1})=w f_n$ for all $n\in\N$.
Hence by Equation (\ref{e:page3.9}),
$$\lim_{k\to\infty} ||(u_k-w) f_n||=0,\quad n\in\N.$$
Therefore,
$$\lim_{k\to\infty} ||(u_k-w) f_n||_2=0,\quad n\in\N.$$
Let $\varepsilon>0$ and choose $n$ such that $\tau(f_n^\perp)<\varepsilon$.
By Lemma \ref{l:3.5}(3)
$$||(u_k-w)f_n^\perp||_2\le 2||f_n^\perp||_2<2\varepsilon^{1/2}.$$
So by Lemma \ref{l:3.5}(2)
$$\limsup_{k\to\infty} ||u_k-w||_2<2\varepsilon^{1/2}.$$
Hence $u_k$ converges to $w$ in the $d_\tau$-metric, proving the completeness of $U(M)$ in the $d_\tau$-norm.
We next prove that the self-adjoint part of the unit ball $(M_{\mathrm{sa}})_1$ is complete in the $d_\tau$-metric:
Let $a_n$ be a $d_\tau$-Cauchy sequence in $(M_{\mathrm{sa}})_1$.
Denote by $u_n$ the Cayley transform of $a_n$:
$$u_n=(a_n+i1)(a_n-i1)^{-1}\in U(M).$$
Then
$$u_n-u_m=2(a_n-i1)^{-1}(a_m-a_n)(a_m-i1)^{-1},$$
so by Lemma \ref{l:3.5}(3)
$$||u_n-u_m||_2\le 2||a_m-a_n||_2.$$
Thus, $u_n$ converges in the $d_\tau$-metric to a unitary $u\in U(M)$.
Since $\sigma(a_n)\subset [-1,1]$, $\sigma(u_n)\subset\{t\in\T\mid \mathrm{Re}\, t\le 0\}$.
Hence $||1+u_n||\le\sqrt 2$, and therefore $||1+u||\le \sqrt 2$ by Lemma \ref{l:3.8}, i.e.~$\sigma(u)\subset\{t\in\T\mid \mathrm{Re }\,t\le0\}$.
Let
$$a=i(u+1)(u-1)^{-1}$$
be the inverse Cayley transform of $u$. Then
$$\sigma(a)\subset [-1,1].$$
Hence $a\in (M_{\mathrm{sa}})_1$. Since $a_n=i(u_n+1)(u_n-1)^{-1}$, we have
$$a_n-a=2i(u_n-1)^{-1}(u-u_n)(u-1)^{-1}.$$
By the conditions on $\sigma(u_n)$ and $\sigma(u)$,
$$||(u_n-1)^{-1}||\le \frac 1{\sqrt 2}\quad\text{and}\quad ||(u-1)^{-1}||\le \frac 1{\sqrt 2}.$$
Hence, using Lemma \ref{l:3.5}(3)
$$||a_n-a||_2\le ||u-u_n||_2\to 0 \text{ for } n\to\infty$$
proving the $d_\tau$-completeness of $(M_{\mathrm{sa}})_1$. Finally if $x_n$ is a $d_\tau$-Cauchy net in $M_1$, the closed unit ball of $M$, then $a_n=\frac 12(x_n+x_n^*)$ and $b_n=\frac 1{2i}(x_n-x_n^*)$ are $d_\tau$-Cauchy nets in $(M_{\mathrm{sa}})_1$ by Lemma \ref{l:3.5}(2). Hence by the $d_\tau$-completeness of $(M_{\mathrm{sa}})_1$ and by Lemma \ref{l:3.5}(2), $x_n=a_n+ib_n$ is convergent in the $d_\tau$-metric. Moreover, by Lemma \ref{l:3.8}, the limit is also in $M_1$. This completes the proof.
\end{proof}
We need the following version of Kaplansky's Density Theorem.
\begin{lem}\label{l:3.11}
  Let $A$ be a unital $C^*$-algebra with a faithful quasitrace $\tau$ and let $B$ be a unital $C^*$-subalgebra.
  Then the following two conditions are equivalent.
  \begin{itemize}
    \item[($1$)] $B$ is dense in $A$ in the $d_\tau$-metric.
    \item[($2$)] $B_1$ is dense in $A_1$ in the $d_\tau$-metric.
  \end{itemize}
  Here $A_1$ and $B_1$ denote the norm-closed unit balls  of $A$ and $B$, respectively.
\end{lem}

\begin{proof}
$(2)\Rightarrow(1)$: trivial.
$(1)\Rightarrow(2)$: This follows essentially the proof of the ``classical'' Kaplansky theorem:
Consider the real function
$$f(t)=\frac{2t}{1+t^2},\quad t\in \R.$$
Then $|f(t)|\le 1$ for all $t\in \R$, and the restriction of $f$ to $[-1,1]$ is a homeomorphism of $[-1,1]$.
Let
$$g:[-1,1]\to[-1,1]$$
be the inverse of this function. Note that
\begin{eqnarray*}
  &&f(-t)=-f(t),\quad t\in \R\\
  &&g(-t)=-g(t),\quad t\in[-1,1].
\end{eqnarray*}
Assume (1), and let $x\in A_1$. Let
$$a=\left(
      \begin{array}{cc}
        0 & x \\
        x^* & 0 \\
      \end{array}
    \right)\in
(M_2(A)_{\mathrm{sa}})_1.
$$
Since $g$ is an odd function, $b=g(a)$ is of the form
$$ b=g(a)=\left(
            \begin{array}{cc}
              0 & y \\
              y^* & 0 \\
            \end{array}
          \right)
$$
for some $y\in A_1$. Moreover, since $a=f(b)$,
$$x=2y(1+y^*y)^{-1}=2(1+yy^*)^{-1}y.$$
Choose a sequence $y_n\in A$ such that $||y_n-y||_2\to 0$, and let
$$x_n=2y_n(1+y_n^*y_n)^{-1}\in B.$$
Then $x_n^*x_n=4y_n^*y_n(1+y_n^*y_n)^{-2}\le 1$ because $\sup_{s\ge 0} 4s(1+s)^{-2}=1$. Hence $x_n\in B_1$. Moreover,
\begin{eqnarray*}
  x_n-x&=&2(1+yy^*)^{-1}((1+yy^*)y_n-y(1+y_n^*y_n))(1+y_n^*y_n)^{-1}\\
  &=&2(1+yy^*)^{-1}(y_n-y)(1+y_n^*y_n)^{-1}+2(1+yy^*)^{-1}y(y^*-y_n^*)y_n(1+y_n^*y_n)^{-1}.
\end{eqnarray*}
Since $(1+yy^*)^{-1}$, $(1+y_n^*y_n)^{-1}$, $2(1+yy^*)^{-1}y$ and $2y_n(1+y_n^*y_n)$ all have $C^*$-norm at most 1,
Lemma \ref{l:3.5} yields
\begin{eqnarray*}
   ||x_n-x||_2^{2/3}&\le& 2^{2/3}||y_n-y||_2^{2/3}+2^{-2/3}||y_n^*-y^*||_2^{2/3}\\
   &=&(2^{2/3}+2^{-2/3})||y_n-y||_2^{2/3}\\
   &\to& 0 \text{ for } n\to\infty.
\end{eqnarray*}
Hence $x$ is in the $d_\tau$ closure of $B_1$.
\end{proof}

\begin{pro}\label{p:3.12}
  Let $M$ be a finite $AW^*$-algebra with a faithful normal quasitrace $\tau$ and let $A$ be a unital $C^*$-subalgebra of $M$.
  Then the $d_\tau$-closure of $A$ in $M$ is the smallest $AW^*$-subalgebra of $M$ containing $A$.
\end{pro}

\begin{proof}
  Let $B$ be the $d_\tau$-closure of $A$. By Lemma \ref{l:3.7}, $B$ is a unital $C^*$-subalgebra of $M$ (note that norm-convergence implies $\tau$-convergence). By Lemma \ref{l:3.8} and Lemma \ref{l:3.11}, $B_1$ is the $d_\tau$-closure of $A_1$.
  Hence by Proposition \ref{p:3.10} applied to $M$, $B_1$ is complete in the $d_\tau$-metric so by Lemma \ref{l:3.9}, $B$ is an $AW^*$-algebra in its own right. To be a $AW^*$-subalgebra however, one also requires  that if $(p_i)_{i\in I}$ is a set of orthogonal projections in $B$ and $p=\bigvee_{i\in I} p_i$ is the least upper bound of $(p_i)_{i\in I}$ computed in the projection lattice of $M$,  then $p\in B$.
  However, this is clearly true because $p$ is the $d_\tau$-limit of the net $(\sum_{i\in J})_{J\in \fJ}$, where $\fJ$ is the family of finite subsets of $I$ (cf.~proof of Lemma \ref{l:3.11}). Hence $B$ is an $AW^*$-subalgebra of $M$. Conversely, if $C$ is an $AW^*$-subalgebra of $M$ containing $A$, then by Proposition \ref{p:3.10}, $C_1$ is $d_\tau$-complete, so by Lemma \ref{l:3.11} $C$ is $d_\tau$-closed. Hence, $C\supset B$.
\end{proof}

\section{Ultraproducts and $AW^*$-completions}
The following lemma is probably well known. For completeness we include a proof.

\begin{lem}\label{l:4.1}
  Let $(X_n,d_n)_{n=1}^\infty$ be a sequence of metric spaces for which
  $$\sup\limits_{n\in\N} \mathrm{diam}(X_n) <\infty.$$
  Let $\fU$ be a free ultrafilter on $\N$. Define an equivalence relation $\sim$ on $X=\prod\limits_{n=1}^\infty X_n$ by
  $$x\sim y \iff \lim\limits_{\fU} d_n(x_n,y_n)=0.$$
  Then $X/\sim$ is a complete metric space in the metric
  $$d([x],[y])=\lim\limits_{\fU} d_n(x_n,y_n).$$
\end{lem}
\begin{proof}
  Define
  $$\bar d(x,y)=\lim\limits_{\fU} d_n(x_n,y_n),\quad x,y\in X.$$
  Then $\bar d$ induces a metric on $X/\sim$ by
  $$d([x],[y])=\bar d(x,y).$$
  Let $(z_i)_{i=1}^\infty$ be a Cauchy sequence in $X/\sim$.
  To prove convergence of $(z_i)_{i=1}^\infty$, it suffices to prove that $(z_i)_{i=1}^\infty$  has a convergent subsequence. Hence we may assume
  $$d(z_i,z_{i+1})<2^{-i}, \quad i\in\N.$$
  Choose $x^{(i)}=(x_n^{(i)})_{n\in\N}$ in $X$, such that $z_i=[x^{(i)}]$. Since
  $$\lim_{n,U} d(x_n^{(i)},x_n^{(i+1)})< 2^{-i}$$
  we can choose sets
  $$F_1\supset F_2\supset \cdots \supset F_i\supset$$
  in $\fU$ such that
  $$d(x_n^{(i)},x_n^{(i+1)})< 2^{-i},\quad \forall n\in F_i.$$
  Since $\fU$ is free we can replace $F_i$ with $F_i\cap \{i,i+1,\ldots\}$ and obtain that also
  $$\bigcap_{i=1}^\infty F_i=\varnothing.$$
  Let $F_0=\N$ and note that $\N$ is the disjoint union of $(F_{i-1}\backslash F_i)_{i=1}^\infty$.
  Hence we can define $x=(x_n)_{n=1}^\infty\in X$ by
  $$x_n=x_n^{(i)}, \qquad n\in F_{i-1}\backslash F_i.$$
  Let $n\in F_i$. Then $n\in F_{j-1}\backslash F_j$    for some $j>i$.
  For this $j$,
  \begin{eqnarray*}
    d_n(x_n^{(i)},x_n)&=&d_n(x_n^{(i)},x_n^{(j)})\\
    &\le&\sum_{k=i}^{j-1} d(x_n^{(k)},x_n^{(k+1)})\\
    &<& 2^{1-i}.
  \end{eqnarray*}

  Since $F_i\in \fU$, $d([x^{(i)}],[x])\le \sup\limits_{n\in F_i} d_n(x_n^{(i)},x_n)\le 2^{1-i}$.
  Therefore,
  $z_i=[x^{(i)}]$ converges to $[x]$ in $X/\sim$.
    \end{proof}
  If $(A_n)_{n=1}^\infty$ is a sequence of $C^*$-algebras, we let $\ell^\infty\{A_n\}=\{(x_n)_{n=1}^\infty\mid x_n\in A, \,\,\sup ||x_n||<\infty\}$.
  If $A_n=A$ (fixed) for all $n$, we write $\ell^\infty(A)$ instead.

  \begin{pro}\label{p:4.2}
    Let $(A_n,\tau_n)_{n=1}^\infty$ be a sequence of unital $C^*$-algebras with normalized quasitraces $\tau_n$, and let $\fU$ be a free ultrafilter on $\N$.
    Let
    $$\fJ_{\fU}=\{ (x_n)_{n=1}^\infty\in \ell^\infty\{A_n\} \mid \lim\limits_{\fU} \tau_n(x_n^*x_n)=0  \}.$$
    Then $\fJ_{\fU}$ is a norm-closed two-sided ideal in $\ell^\infty\{A_n\}$, and $\ell^\infty\{ A_n\}/\fJ_{\fU}$ is a finite $AW^*$-algebra with normal faithful quasitrace $\tau_\fU$ given by
    \begin{equation}\label{e:page4.3}
      \tau_\fU([x])=\lim\limits_{\fU} \tau_n(x_n), \quad x=(x_n)_{n=1}^\infty\in \ell^\infty \{A_n\},
    \end{equation}
    where $x\to [x]$ is the quotient map from $\ell^\infty\{A_n\}$ to $\ell^\infty\{A_n\}/\fJ_\fU$.
  \end{pro}
\begin{proof}
  Without loss of generality, assume that each $\tau_n$ is faithful. Otherwise, we can replace $A_n$ with $A_n/I_n$, where
  $$I_n=\{x\in A_n \mid \tau_n(x^*x)=0  \}$$
  (cf.~Proposition \ref{p:3.2}). It is clear that
  $$\bar \tau_n(x)=\lim\limits_{\fU} \tau_n(x_n), \qquad x=(x_n)_{n=1}^\infty\in \ell^\infty\{A_n\}$$
  defines a quasitrace on $\ell^\infty\{A_n\}$, and  so by Proposition \ref{p:3.2}, $\fJ_{\fU}$ is a norm-closed two sided ideal in $\ell^\infty\{A_n\}$ and there is a faithful quasitrace $\tau_\fU$ on $\ell^\infty\{A_n\}/\fJ_\fU$ such that Equation (\ref{e:page4.3}) holds. Since a $*$-homomorphism of a $C^*$-algebra $A$ onto a $C^*$-algebra $B$ maps the closed unit ball of $A$ onto the closed unit ball of $B$,  we get from Definition \ref{d:3.6} and Lemma \ref{l:4.1} that the unit ball of $\ell^\infty\{A_n\}/\fJ_\fU$ is complete in the metric associated with $\tau_\fU$. Hence Lemma \ref{l:3.9} completes the proof of Proposition \ref{p:4.2}.
\end{proof}
  The following is a slight extension of Corollary II.2.4 in \cite{BH}.
  \begin{cor}\label{c:4.3}
    Let $A$ be a unital $C^*$-algebra with a faithful quasitrace $\tau$.  Then there is a one-to-one $*$-homomorphism $\pi$ of $A$ into a finite $AW^*$-algebra $M$ with a faithful normal quasitrace $\bar \tau$ such that
    $$\tau(x)=\bar\tau \circ \pi(x),\qquad x\in A.$$
  \end{cor}
  \begin{proof}
  Let $A_n=A$ for all $n$ and apply Proposition \ref{p:3.2}.
  The $*$-homomorphism $\pi$ is given by
  $$\pi(x)=[(x)_{n=1}^\infty].$$
  \end{proof}

Let $A$ and $M$ be as in Corollary \ref{c:4.3}.
Then by Proposition \ref{p:3.12}, the closure $B$ of $\pi(A)$ in the $d_\tau$-metric is the smallest $AW^*$-subalgebra of $M$ containing $A$. Moreover, by Lemma \ref{l:3.11}, every element of $B$ is the $d_{\bar \tau}$-limit of a bounded sequence in $\pi(A)$. Since for every $t>0$, the $t$-ball of $B$ is $d_{\bar\tau}$-complete by Proposition  \ref{p:3.10}, $B$ is equal to the quotient $C^*$-algebra
$$B=\tilde A/\tilde I$$
where
$$\tilde A=\{(x_n)_{n=1}^\infty \in \ell^\infty(A)\mid x_n \text{ is a $d_\tau$-Cauchy sequence}   \}$$
and
$$\tilde I=\{(x_n)_{n=1}^\infty\in \ell^\infty(A)\mid x_n\to 0 \text{ in the $d_\tau$-metric}\}.$$
The restriction of $\bar \tau$ to $B=\tilde A/\tilde I$ is given by
\begin{equation}\label{e:page4.5}
  \bar\tau([x])=\lim\limits_{n\to\infty} \tau(x_n),\qquad  x=(x_n)\in \tilde A.
\end{equation}

Indeed, Equation (\ref{e:page4.5}) follows from Lemma \ref{l:3.7}(4) when $x\ge 0$  and by Definition \ref{d:3.1} $(ii)$ and $(iii)$ for general $x\in \tilde A$. In particular, we have
\begin{pro}\label{p:4.4}
  Let $A$ be a unital $C^*$-algebra with a faithful quasitrace $\tau$. Let $(\pi,M,\bar\tau)$ and $(\pi^1,M^1,\bar\tau^1)$ be two triples satisfying the conditions of Corollary \ref{c:3.4} and let $B$ (resp.~$B^1$) denote the $AW^*$-subalgebra of $M$ (resp.~$M^1$) generated by $\pi(A)$ (resp.~$\pi^1(A)$).
  Then there is a unique $*$-isomorphism
  $$\rho:B\twoheadrightarrow B^1$$
  such that $\pi^1=\rho \circ\pi$ and $\bar \tau=\bar\tau^1\circ \rho$.
\end{pro}

\begin{proof}
  With the notation preceding Proposition \ref{p:4.4}, both $B$ and $B^1$ are naturally isomorphic to $\tilde A/\tilde I$.
\end{proof}

\begin{defn}\label{d:4.5}
  Let $A$ be a unital $C^*$-algebra with a faithful quasitrace $\tau$, and let $B=\tilde A/\tilde I$ be the finite $AW^*$-algebra described prior to Proposition \ref{p:4.4} with normal faithful quasitrace
  $$\bar\tau ([x])=\lim\limits_{n\to\infty} \tau(x_n).$$
  Let us call $(B,\bar\tau)$ the $AW^*$-completion of $(A,\tau)$.
\end{defn}

\begin{pro}\label{p:4.6}
  Let $\tau$ be a faithful normalized quasitrace on a unital $C^*$-algebra $A$. If $\tau$ is an extreme point in $QT(A)$, then the $AW^*$-completion of $(A,\tau)$ is a finite $AW^*$-factor.
\end{pro}

\begin{proof}
  The $W^*$-version of this is well known, and the proof for the above case is the same: Indeed, if the $AW^*$-completion $(B,\bar \tau)$ is not a factor, then choose a central projection $p\in B$, $p\ne 0$, $p\ne 1$.
  Let $\pi$ be the embedding of $A$ into $B$. Since $\pi(A)$ is $d_{\bar\tau}$-dense in $B$ it follows easily that $\tau=\tau_1+\tau_2$, where $\tau_1,\tau_2$ are the quasitraces
  \begin{eqnarray*}
    &&\tau_1(x)=\bar\tau(p\pi(x)), \qquad x\in A\\
    &&\tau_2(x)=\bar\tau((1-p)\pi(x)),\qquad x\in A.
  \end{eqnarray*}
  and $\tau_1\ne0$, $\tau_2\ne0$. By normalizing $\tau_1$ and $\tau_2$ we get that $\tau$ is a non-trivial convex combination of elements from $QT(A)$, which contradicts that $\tau$ is extreme.
\end{proof}

\section{The main result}
Recall that a $C^*$-algebra $A$ is exact if for all pairs $(B,J)$ of a $C^*$-algebra $B$ and a closed two sided ideal $J$ in $B$, the sequence
$$\xymatrix{0\ar[r]&A\otimes J\ar[r]&A\otimes B \ar[r]&A\otimes B/J\ar[r]&0}
$$
is exact. Here the tensor product $\otimes$ denotes the spatial\,(=minimal) tensor product of the $C^*$-algebras involved (cf.~\cite{K1}).
It is well known, that nuclear $C^*$-algebras and subalgebras of nuclear algebras are exact. Recently in \cite{K2}, Kirchberg proved that the class of exact $C^*$-algebras coincide with the class of quotients of subalgebras of nuclear $C^*$-algebras. In particular,
\begin{pro}[\cite{K2}]\label{p:5.1}
  Any quotient $C^*$-algebra of an exact $C^*$-algebra is exact.
\end{pro}

\begin{pro}\label{p:5.2}
  $C^*_r(\F_n)$ is an exact $C^*$-algebra for any $n\in\N$, $n\ge 2$ and for $n=\infty$.
\end{pro}

\begin{proof}
  This is well known. The case $n=2$ is in \cite{EH} and the general case follows easily because $\F_n$ can be embedded in $\F_2$ for all $n\ge 2$ including $n=\infty$. One can also use $\S 6$ in \cite{DCH} to get that $C_r^*(\Gamma)$ is exact for any discrete subgroup $\Gamma$  of $SL(2,\R)$, in particular for $\Gamma=\F_n$.
\end{proof}\label{r:5.3}
\begin{rem} More generally $C_r^*(\Gamma)$ is exact for any closed  discrete subgroup $\Gamma$ of an almost connected locally compact group (cf.~\cite{KW}).
\end{rem}

\begin{defn}\label{d:5.4}
  For any free ultrafilter $\fU$ on $\N$, set
  $$I_\fU=\{(x_n)_{n=1}^\infty\in\ell^\infty\{M_n(\C)\}\mid \lim\limits_{\fU}\mathrm{tr}_n(x_n^*x_n)=0\},$$
  where $\mathrm{tr}_n$ is the normalized trace on $M_n(\C)$, and let
  $$M_{\fU}=\ell^\infty\{M_n(\C)\}/I_\fU.$$
\end{defn}
It is well known that $M_\fU$ is a $II_1$-factor with normal trace:
$$\tau_n([x])=\lim\limits_{\fU} \tau_n(x_n),\qquad x=(x_n)_{n=1}^\infty\in\ell^\infty\{M_n(\C)\}.$$
We shall need the following result of Wassermann.
\begin{pro}[\cite{W}]\label{p:5.5}
  Let $\Gamma$ be a residually finite countable discrete ICC-group. Then the $II_1$-factor $L(\Gamma)$ associated with the left regular representation of $\Gamma$ is isomorphic to a subfactor of $M_\fU=\ell^\infty\{M_n(\C)\}/I_\fU$ for some free ultrafilter $\fU$ on $\N$.
  In particular, $L(\F_n)$ has this property for $n=2,3,\ldots$ and $n=\infty$.
  (As usual, $\F_n$  denotes the free group on $n$ generators.)
\end{pro}
\begin{lem}\label{l:5.6}
  Let $\tau$ be a normalized quasitrace on a unital $C^*$-algebra $A$ and let $\tau_n$ be the (unique) quasitrace on $M_n(A)$ for which $\tau_n(x)=\tau(x\otimes e_{11})$, $x\in A$. Let
  $$\tau_n'(x)=\frac 1n \tau(x), \qquad x\in A.$$
  Then
  \begin{itemize}
    \item[(1)] $\tau'_n(x\otimes 1_n)=\tau(x),\qquad x\in A$
    \item[(2)] $\tau'_n(1\otimes y)=\mathrm{tr}_n(y),\qquad y\in M_n(\C)$.
  \end{itemize}
  Moreover, if $\tau$ is faithful, then  $\tau_n'$ is also faithful.
\end{lem}
\begin{proof}
  From Definition \ref{d:3.1}(1) we have
  \begin{equation}\label{e:page5.3}
    \tau(uau^*)=\tau(a), \quad a\in A_n, u\in U(A)
  \end{equation}
  where $U(A)$ is the unitary group of $A$. By Definition \ref{d:3.1}(2), Equation (\ref{e:page5.3}) holds for all $a\in A_{\mathrm{sa}}$.
  Hence
  $$\tau_n(a\otimes e_{kk})=\tau_n(a\otimes e_{11})=\tau(a),\quad a\in A_{\mathrm{sa}}.$$
and since $(a\otimes e_{kk})_{k=1}^n$ are contained in an abelian $C^*$-subalgebra of $M_n(A)$
$$\tau_n'(a\otimes 1_n)=\frac 1n \sum_{k=1}^n \tau_n(a\otimes e_{kk}) =\tau(a), \qquad a\in A_{\mathrm{sa}}.$$
By Definition \ref{d:3.1}(3) this can be extended to all $a\in A$, proving (1).
(2) holds because $\mathrm{tr}_n$ is the unique normalized quasitrace on $M_n(\C)$. Assume next that $\tau$ is faithful on $A$, and let
$$x=\sum_{i,j=1}^n x_{ij}\otimes e_{ij}$$
  be an element of $M_n(A)$ for which $\tau'_n(x^*x)=0$.
  By Lemma \ref{l:3.5}(3) also
  $$||x_{ij}\otimes e_{11}||_2=||(1\otimes e_{1i})x (1\otimes e_{j1})||_2=0$$
  where $||z||_2=\tau_n'(z^*z)^{1/2}$ for $z\in M_n(A)$. Hence
  $$\tau(x_{ij}^*x_{ij})=n\tau_n'(x_{ij}^*x_{ij}\otimes e_{jj})=0,\qquad 1\le 1,j\le n.$$
  Thus $x_{ij}=0$ for all $i,j$ proving $x=0$.
  Hence $\tau'_n$ is faithful.
\end{proof}

\begin{lem}\label{l:5.7}
  Let $A$ be a unital exact $C^*$-algebra with a faithful normalized quasitrace $\tau$. Then for any free ultrafilter $\fU$ on $\N$, the spatial $C^*$-tensor product
  $$A\otimes_{\min} M_\fU$$
  can be embedded in a finite $AW^*$-algebra $N$  with a faithful normal quasitrace $\bar\tau$ for which
  \begin{eqnarray*}
    &&\bar\tau (x\otimes 1)=\tau(x), \qquad x\in A\\
    &&\bar \tau(1\otimes y)=\tau_\fU(y),\qquad y\in M_\fU.
  \end{eqnarray*}
\end{lem}

\begin{proof}
  Let $N=\ell^\infty\{M_n(A)\}/\fJ_\fU$, where
  $$\fJ_\fU=\{(x_n)_{n=1}^\infty \in \ell^\infty\{M_n(A)\} \mid \lim \limits_{\fU} \tau'_n (x_n^*x_n)=0 \}.$$

By Proposition \ref{p:4.2} and Lemma \ref{l:5.6}, $N$ is a finite $AW^*$-algebra with faithful normal quasitrace $\bar \tau$  given by
$$\bar \tau ([x])=\lim\limits_{\fU} \tau'_n(x),\qquad x=(x_n)_{n=1}^\infty\in \ell^\infty(M_n(A)),$$
where $z\to [z]$ is the quotient map from $\ell^\infty\{M_n(A)\}$ to $N$.
Define  a unital *-homomorphism  $\pi:A\to N$ by
$$\pi(x)=[(x\otimes 1_n)_{n=1}^\infty].$$
By Lemma \ref{l:5.6}(1)
$$\bar\tau\circ \pi(x)=\tau(x), \qquad x\in A.$$
In particular, $\pi$ is one-to-one.
Since by Lemma \ref{l:5.6}(2),
$$\tau'_n(1\otimes y)=\mathrm{tr}_n(y),$$
there is a one-to-one unital *-homomorphism $\rho:M_\fU\to N$ such that
$$\rho([(x_n)_{n=1}^\infty])=[(1\otimes x_n)_{n=1}^\infty],$$
for $(x_n)_{n=1}^\infty\in \ell^\infty\{M_n(\C)\}$, and moreover
$$\bar\tau\circ \rho=\tau_\fU.$$
It is clear that $\pi(A)$ and $\rho(M_\fU)$  are commuting subalgebras of $N$. The map
$$\beta:\sum_{i=1}^\ell x_i\otimes z_i\to ||\sum_{i=1}^\ell \pi(x_i)\rho(z_i)||$$
  defines a $C^*$-semi-norm on the algebraic tensor product $A\odot M_\fU$. To prove that $\beta$ is a norm, it suffices to prove that $\beta(x\otimes z)=0$ implies $x=0$ or $z=0$. (See e.g. the section on tensor products of $C^*$-algebras of Sakai's book \cite{S}). Recall that $M_\fU$ is a $II_1$-factor and therefore a simple unital $C^*$-algebra. Assume $x\in A$, $z\in M_\fU$ and $\beta(x\otimes z)=0$. Since
  $$I=\{w\in M_\fU\mid \pi(x) \rho (w)=0\}$$
  is a two sided ideal in $M_\fU$, either $I=\{0\}$ or $I=M_\fU$. In the first case $z=0$ and in the second case $x=0$ proving that $\beta$ is a $C^*$-norm on $A\odot M_\fU$, so with standard notation for $C^*$-norms on tensor products,
  $$\min \le \beta\le \max.$$

  To prove $\beta=\min$, we need the assumption that $A$ is exact:
  Let $\ell\in \N$, $x_1,\ldots,x_\ell\in A$ and $y_1,\ldots, y_\ell\in \ell^\infty(M_n(\C))$ and let $[y_i]$  be the range of $y_i$ in $M_\fU$ by the quotient map. Write $y_i=((y_i)_n)_{n=1}^\infty$ where $(y_i)_n\in M_n(\C)$. Then
$$\sum_{i=1}^\ell \pi(x_i)\rho([y_i])  =[\left(\sum_{i=1}^\ell x_i\otimes (y_i)_n\right)_{n=1}^\infty],$$
  where as above $[\,\cdot\,]$ on the right side of the equality denotes the quotient map $\ell^\infty(M_n(A))\to N$.
  Hence $$\beta(\sum_{i=1}^\ell x_i\otimes [y_i])\le \sup\limits_{n\in\N} ||\sum_{i=1}^\ell x_i\otimes (y_i)_n||_{\min}=||\sum_{i=1}^\ell x_i\otimes y_i||_{\min}.$$
  Hence the map
  $$\sum_{i=1}^\ell x_i\otimes y_i\to \sum_{i=1}^\ell \pi(x)\rho([y_i]),\qquad x_i\in A, y_i\in \ell^\infty\{M_n(\C)\}$$
  extends to a $*$-homomorphism
  $\phi:A\otimes \ell^\infty\{M_n(\C)\}\to C^*(\pi(A),\rho(M_\fU)).$
  Note that $\beta([z])=||\phi(z)||$, $z\in A\otimes \ell^\infty\{M_n(\C)\}$.
  For $x\in A$, and $y\in \ell^\infty\{M_n(\C)\}$,
  \begin{eqnarray*}
    \bar\tau\circ\phi((x\otimes y)^*(x\otimes y))&=&\bar \tau (\rho([y])^*\pi(x^*x)\rho([y])\\
    &\le& ||x||^2\bar\tau(\rho([y^*y])\\
    &=&||x||^2\lim\limits_{\fU} \mathrm{tr}_n (y_n^*y_n).
  \end{eqnarray*}
 Since $\bar\tau$ is faithful, it follows that $\ker\phi$ contains $A\otimes I_\fU$. Therefore the $C^*$-tensor norm $\beta$ on $A\otimes M_\fU$
  is less or equal to the norm on $A\odot M_\fU$ coming from the quotient
  $$A\otimes \ell^\infty\{M_n(\C)\}/A\otimes I_\fU.$$
  However, exactness of $A$ implies that the latter norm is the minimal $C^*$-tensor norm.
  Hence $\beta\le \min$, so altogether $\beta=\min$.  This shows that the map
  $$\sum_{i=1}^\ell x_i\otimes z_i\to \sum_{i=1}^\ell \pi(x_i)\rho(z_i),\qquad x_i\in A,z_i\in M_\fU$$
  extends to a one-to-one *-homomorphism of $A\otimes M_\fU$ into $N$ with the desired properties.
\end{proof}
\begin{lem}\label{l:5.8}
  Let $N$ be a finite $AW^*$-algebra with a faithful normal quasitrace $\tau$ and let $A$ and $C$ be two commuting unital $C^*$-subalgebras of $N$.
  Let $B$ be the $AW^*$-subalgebra of $N$ generated by $A$.
  If
  \begin{itemize}
    \item[($i$)] $||\sum_{i=1}^\ell a_ic_i||\le ||\sum_{i=1}^\ell a_i\otimes c_I||_{\min}$,  $\ell\in\N$, $a_1,\ldots,a_\ell\in A$, $c_1,\ldots,c_\ell\in C$
          \end{itemize}
          and
  \begin{itemize}
    \item[($ii$)] $C$ is an exact $C^*$-algebra,
  \end{itemize}
  then\\
   $$||\sum_{i=1}^\ell b_ic_i||\le ||\sum_{i=1}^\ell b_i\otimes c_i||_{\min},\qquad \ell\in\N,  b_1,\ldots,b_\ell\in B, c_1,\ldots,c_\ell\in C.$$
\end{lem}
\begin{proof}
  Note first that by Proposition \ref{p:3.12} and Lemma \ref{l:3.11} every element of $B$ is the $d_\tau$-limit of a bounded sequence in $A$. Hence, by Lemma \ref{l:3.7}, $B$ and $C$ also commute. By the remarks prior to Proposition \ref{p:4.4},
  $$B=\tilde A/\tilde I$$
  where
  \begin{eqnarray*}
    &&\tilde A=\{(x_n)\in \ell^\infty (A)\mid x_n \text{ is a $d_\tau$-Cauchy sequence }\}\\
    &&\tilde I=\{(x_n)\in \ell^\infty(A)\mid x_n\to 0 \text{ in the  $d_\tau$-metric } \}
  \end{eqnarray*}
  and the quotient map $\phi:\tilde A\to B$ is given by
  $$\phi((a_n)_{n=1}^\infty)=d_\tau\text{-}\!\!\!\lim\limits_{n\to\infty} a_n.$$
  Since $B$ and $C$ commute, we can define a *-homomorphism
  $$\psi_0:\tilde A \odot C\to C^*(B,C)\subset N$$
  by
  $$\psi_0(\sum_{i=1}^\infty a_i\otimes c_i)=\sum_{i=1}^\ell \phi(a_i) c_i.$$
  Since $\phi(a_i)=d_\tau\text{-}\!\!\!\lim\limits_{n\to\infty}(a_i)_n$ we get from Lemma \ref{l:3.7} that
  $$\sum_{i=1}^\ell \phi(a_i)c_i=d_\tau\text{-}\!\!\!\lim\limits_{n\to\infty}\sum_{i=1}^\infty (a_i)_nc_i.$$
  By Lemma \ref{l:3.8}, the $t$-ball  of $N$
  $$N_t=\{x\in\N\mid ||x||\le t\}$$
  is closed in the $d_\tau$-metric for all $t>0$. Hence
  $$||\sum_{i=1}^\ell \phi(a_i)c_i||\le \sup\limits_{n\in\N} ||\sum_{i=1}^\ell (a_i)_nc_i||$$
  and therefore by Condition $(i)$ in the lemma,
  $$\sum_{i=1}^\ell \phi(a_i)c_i||\le \sup\limits_{n\in\N}|| \sum_{i=1}^\ell (a_i)_n\otimes c_i||_{\min}=||\sum_{i=1}^\ell a_i\otimes c_i||_{\min}$$
  where the first $||\cdot||_{\min}$ is in $A\otimes C$ and the second in $\tilde A \otimes C$. The last equality follows from the inclusions
  $$\tilde A \otimes C\subset \ell^\infty(A)\otimes C\subset \ell^\infty(A\otimes C).$$
  This shows that $\psi_0$ extends to a *-homomorphism
  $$\psi:\tilde A\otimes C \to C^*(B,C).$$
  The kernel of $\psi$ clearly contains
  $$\ker \phi \otimes C=\tilde I \otimes C.$$
  Since $C^*(B,C)=(\tilde A\otimes C)/\ker \phi$ the $C^*$-seminorm on $B\odot C$ inherited from $C^*(B,C)$ is dominated by the $C^*$-norm on $B\otimes C$ coming from
  $$(\tilde A\otimes C)/(\tilde I \otimes C).$$
  However, by exactness of $C$, the latter norm is equal to the minimal tensor norm on $B\odot C$. This proves Lemma \ref{l:5.8}.
\end{proof}

\begin{rem}\label{r:5.9}
  Kirchberg has proved that exactness for a $C^*$-algebra is equivalent to the properties $C$ and $C'$ of Archbold and Batty (see Section 7 in \cite{K2} and \cite{AB}). Lemma \ref{l:5.8} can be considered as an $AW^*$-analogue of the implication exact $\Rightarrow$ property $C'$.
\end{rem}

\begin{lem}\label{l:5.10}
  Let $A$ be a unital exact $C^*$-algebra with a faithful quasitrace $\tau$. Let $M_\tau$ be the $AW^*$-completion of $A$ with respect to the $\tau$ (cf.~ Definition \ref{d:4.5}).
  Then
  $$M_\tau\otimes C_r^*(\F_\infty)$$
  can be embedded in a finite $AW^*$-algebra.
\end{lem}
\begin{proof}
  Let $\fU$ be a free ultrafilter on $\N$.
  By Lemma \ref{l:5.7} $A\otimes M_\fU$ can be embedded in a finite $AW^*$-algebra $N$ with a faithful trace $\bar\tau$ such that
  $$\bar\tau(x\otimes 1)=\tau(x), \qquad x\in A.$$

  Note that $C_r^*(\F_\infty)\subset L(\F_\infty)$, the von Neumann algebra associated with the left regular representation of $\F_\infty$. Since $L(\F_\infty)$ has a unital embedding in $M_\fU$ for some ultrafilter $\fU$ on $\N$ (Proposition \ref{p:5.5}) we get for this $\fU$ that $A\otimes C_r^*(\F_\infty)$ embeds in a finite $AW^*$-algebra $N$ such that
  $$\tau(x)=\bar \tau (x\otimes 1),\qquad x\in A$$
  where $\bar \tau$  is a faithful normal quasitrace on $N$.
  But the $AW^*$-completion $M_\tau$ of $A$ with respect to $\tau$  coincides with the smallest $AW^*$-subalgebra of $N$ containing $A$ (cf.~Proposition \ref{p:4.4} and Definition \ref{d:4.5}). Since $C_r^\star(\F_\infty)$  is exact it follows from Lemma \ref{l:5.8} that
  $$||\sum_{i=1}^\ell a_ib_i||\le ||\sum_{i=1}^\ell a_i\otimes b_i||_{\min}$$
  for all $a_{1},\ldots,a_\ell\in M_\tau$ and $b_1,\ldots,b_\ell\in C_r^*(\F_\infty)$.
  Since $C_r^*(\F_\infty)$ is simple by \cite{AO}, we get as in the proof of Lemma \ref{l:5.7} that $||\sum_{i=1}^\ell a_ib_i||$ defines a $C^*$-norm on $M_\tau\odot C_r^*(\F_\infty)$. Hence
  $$||\sum_{i=1}^\infty a_ib_i||=||\sum_{i=1}^\infty a_i\otimes b_i||_{\min},$$
  proving Lemma \ref{l:5.10}.

\end{proof}
\begin{thm}\label{t:5.11}
  Quasitraces on exact unital $C^*$-algebras are traces.
\end{thm}

\begin{proof}
  Let $\tau$ be an extreme point in the compact convex set $QT(A)$ of normalized quasitraces and let
  $$I=\{x\in A\mid \tau(x^*x)=0\}.$$
  Then, $I$ is a norm-closed two sided ideal (cf.~Proposition \ref{p:3.2}) and
  $$\tau(x)=\tau_0(\rho(x))$$
  for a faithful extremal quasitrace $\tau_0$ on $A/I$, where $\rho:A\to A/I$ is the quotient map.
  Moreover, by Proposition \ref{p:4.6}, the $AW^*$-completion of $A/I$ with respect to $\tau_0$  is a $II_1$-$AW^*$-factor $M_{\tau_0}$
  with a (unique) normal faithful quasitrace $\bar\tau_0$ extending $\tau_0$.
  Assume $\tau$ is not linear.
  Then $\bar\tau_0$ fails to be linear. But uniqueness of the dimension function  on a $II_1$-$AW^*$-factor shows that $\bar\tau_0$ is the only normalized quasitrace on $M_{\tau_0}$. In particular, $M_{\tau_0}$  has no trace states.
  Then by Theorem \ref{t:2.4}
  $$M_{\tau_0}\otimes C_r^*(\F_\infty)$$
  has a non-unitary isometry.
  Since $A/I$ is also exact (Proposition \ref{p:5.1}), this contradicts Lemma \ref{l:5.10}. Hence $\tau$ is linear. By the Krein-Milman Theorem, it now follows that  all $\tau\in QT(A)$ are linear.
\end{proof}

\begin{cor}\label{c:5.12}
  Every stably finite unital exact $C^*$-algebra $A$ has a trace state.
\end{cor}
\begin{proof}
  By \cite{BH} $A$ has a normalized quasitrace.
\end{proof}

\begin{cor}\label{c:5.13}
  If an $AW^*$-$II_1$-factor $M$ is generated (as an $AW^*$-algebra) by an exact unital $C^*$-subalgebra $A$, then $M$ is a von Neumann algebra.
\end{cor}

\begin{proof}
  Let $\tau$ be the unique quasitrace on $M$.
  Then $\tau$ coincides with the dimension function on the projections of $M$. Hence $\tau$ is normal. By Proposition \ref{p:3.12} $A$ is $d_\tau$-dense in $M$, so by Theorem \ref{t:5.11} and Lemma \ref{l:3.7}(4), $\tau$ is additive on $M_+$ and thus linear on $M$.
  Hence by Corollary 7 in \cite{JW}, $M$ is a von Neumann $II_1$-factor. (Note that the last conclusion also follows from Proposition \ref{p:3.10} because completeness of the unit ball of $M$ in the $||\cdot||_2$-norm associated with $\tau$ implies that the range of $M$ by the G.N.S.-representation is a von Neumann algebra.)
\end{proof}

\subsection*{Acknowledgments.}
This paper was completed during the author's stay at the Fields Institute in the spring of 2014.
The author would like to thank George Elliott for his help and support during this stay and Jemima Merisca and Patrina Seepersaud for typesetting the manuscript.

\end{document}